\documentclass[11pt]{amsart}
\usepackage{soul}
\usepackage[latin1]{inputenc}
\usepackage{amsmath, amsfonts, amssymb, amsthm, mathtools, mathrsfs,xypic,bm}
\usepackage{enumerate}
\usepackage{setspace}
\usepackage[numbers]{natbib}

\usepackage[pagebackref,colorlinks,linkcolor=red,citecolor=blue,urlcolor=blue,hypertexnames=true]{hyperref}

\newtheorem{theorem}{Theorem}[section]
\newtheorem{lemma}[theorem]{Lemma}
\newtheorem{proposition}[theorem]{Proposition}
\newtheorem{corollary}[theorem]{Corollary}

\theoremstyle{definition}
\newtheorem{definition}[theorem]{Definition}
\newtheorem{example}[theorem]{Example}
\newtheorem{remark}[theorem]{Remark}

\newcommand{\ep}{\varepsilon}

\newcommand{\N}{\mathbb{N}}
\newcommand{\inv}{{-1}}
\newcommand{\T}{\mathbb{T}}
\newcommand{\Z}{\mathbb{Z}}

\newcommand{\R}{\mathbb{R}}
\newcommand{\C}{\mathbb{C}}

\newcommand{\K}{\mathcal{K}}

\newcommand{\Cuntz}[1]{\mathcal{O}_{#1}}

\newcommand{\Aut}[1]{{\rm Aut}(#1)}

\newcommand{\Prim}[1]{{\rm Prim}(#1)}

\begin{document}

\title{Connective $C^*$-algebras}
\author{Marius Dadarlat}\address{MD: Department of Mathematics, Purdue University, West Lafayette, IN 47907, USA}\email{mdd@purdue.edu}
\author{Ulrich Pennig}\address{UP:
   School of Mathematics,  Cardiff University, Senghennydd Road,  Cardiff CF24 4AG, UK}\email{PennigU@cardiff.ac.uk}	

\begin{abstract}
 Connectivity is a homotopy invariant property of  separable $C^*$-algebras  which has three notable consequences: absence of nontrivial projections, quasidiagonality and a more geometric realization of KK-theory for nuclear $C^*$-algebras using asymptotic morphisms. The purpose of this paper is to further explore the class of connective $C^*$-algebras. We give new characterizations of connectivity for exact and for nuclear  separable $C^*$-algebras and show that an extension of connective separable nuclear $C^*$-algebras is connective. We establish connectivity or lack of connectivity for $C^*$-algebras associated to certain classes of groups: virtually abelian groups, linear connected  nilpotent Lie groups and linear connected semisimple Lie groups.
\end{abstract}
\thanks{M.D. was partially supported by NSF grant \#DMS--1362824}

\maketitle
\section{Introduction}
Connectivity of separable $C^*$-algebras was introduced in our earlier paper \cite{Dad-Pennig-homotopy-symm} under different terminology, see Definition~\ref{def:connectivity} below. The initial motivation for studying it stemmed from our search for homotopy-symmetric $C^*$-algebras.  By a result of Loring and the first author \cite{DadLor:unsusp}, these are precisely the separable $C^*$-algebras for which one can unsuspend in the E-theory of Connes and Higson \cite{Con-Hig:etheory}. Using a result of Thomsen \cite{Thomsen:discrete}, we proved in \cite{Dad-Pennig-homotopy-symm}  that connectivity is equivalent to homotopy-symmetry for all separable nuclear $C^*$-algebras. Moreover, we showed that connectivity has a number of important permanence properties. These facts allowed us to exhibit  new classes of homotopy-symmetric $C^*$-algebras.

 The purpose of this paper is to further explore the class of connective $C^*$-algebras. We are motivated by the following three properties they share:

  (i) If $A$ is a separable nuclear $C^*$-algebra, then $KK(A,B)$ is isomorphic to the homotopy classes of completely positive and contractive (cpc) asymptotic morphisms from $A$  to $B \otimes \K$ for any separable $C^*$-algebra $B$.

 (ii)  Connective $C^*$-algebras are quasidiagonal. In fact, if $A$ is connective, then $A\otimes B$ is quasidiagonal for any $C^*$-algebra $B$. ($A\otimes B$ denotes the  minimal tensor product.)

 (iii) Connective $C^*$-algebras do not have nonzero projections. In fact, if $A$ is connective, then $A\otimes B$ does not have nonzero projections for any $C^*$-algebra $B$.

Connectivity is of particular interest in the case of group $C^*$-algebras.
A countable discrete group $G$ is called connective if the augmentation ideal $I(G)$ defined as the kernel of the trivial representation $\iota \colon C^*(G)\to \C$ is a connective $C^*$-algebra. In view of properties (ii) and (iii) connectivity of $G$ may be viewed as a stringent topological property that accounts simultaneously for the quasidiagonality of $C^*(G)$ and the verification of the Kadison-Kaplansky conjecture for certain classes of groups. Examples of nonabelian connective groups were exhibited in \cite{Dad-Pennig-homotopy-symm} and \cite{Dad-Pennig-Schneider}.

In this paper we give new characterizations of connectivity for  exact and nuclear separable $C^*$-algebras, see Prop.~\ref{products:exact}, \ref{products}. We prove that connectivity of separable nuclear $C^*$-algebras is preserved under extensions, see Thm.~\ref{thm:extensions}. This is a key permanence property which does not hold for quasidiagonal $C^*$-algebras.

There is a close connection between the topology of the spectrum and connectivity, which we employ to reveal an obstruction to connectivity by using work of  Blackadar and Cuntz \cite{BlaCuntz:infproj} and Pasnicu and R\o{}rdam \cite{Pasnicu-Rordam}. In particular, we show that a countable discrete  group $G$ is not connective if the trivial representation $\iota$ is a shielded point of the unitary dual of $G$ in the sense of Def.~\ref{def:shielded}.

Motivated by this, we give a complete description of the neighborhood of $\iota$ in the spectrum of the Hantzsche-Wendt group $G$, which is a torsion free crystallographic group with holonomy $\Z/2\Z \times \Z/2\Z$, in Sec.~\ref{sec:spectrum-HW}. This allows us to prove that $G$ is not connective  in this case (Cor.~\ref{cor:HW_not_connective}).
 Moreover, we show that this group provides a counterexample to a conjecture from \cite{Dadarlat-almost}. Specifically, we prove that the natural map $[[I(G),\K]]\to K^0(I(G))$ is not an isomorphism (Lem.~\ref{lemma:reduction}). In contrast, we show that all torsion free crystallographic groups with \emph{cyclic holonomy} are connective (Thm.~\ref{thm:cyclic_holonomy}).

Next, we investigate connectivity for $C^*$-algebras associated to Lie groups. We show that all noncompact linear connected nilpotent Lie groups have connective $C^*$-algebras (Thm.~\ref{thm:nilpotent}).
Using classic results from representation theory in conjunction with permanence properties of connectivity, we show that  if $G$ is a linear connected  complex semisimple Lie group, then $C^*_r(G)$ is connective if and only if $G$ is not compact (Thm.~\ref{thm:complexLie}). Moreover, if $G$ is a linear connected real reductive Lie group, then $C^*_r(G)$ is connective if and only if $G$ does not have a compact Cartan subgroup (Thm.~\ref{thm:real_reductive}).

A common denominator of our results concerning group $C^*$-algebras is that in all the cases we analyzed, $C^*(G)$  contains a large connective ideal.

\section{Connective $C^*$-algebras}
\subsection{Definitions and background}
For a $C^*$-algebra $A$, the \emph{cone over} $A$ is defined as $CA = C_0[0,1) \otimes A$
the \emph{suspension of} $A$ as $SA= C_0(0,1) \otimes A$.

The first of the following two notions was introduced in \cite[Def.\ 2.6 (i)]{Dad-Pennig-homotopy-symm}, under a different terminology which we have abandoned.  We use the abbreviation \emph{cpc map} for a completely positive and contractive  map.
\begin{definition}\label{def:connectivity}
Let $A$ be a $C^*$-algebra.\\
(a)  $A$ is \emph{connective} if there is a  $*$-monomorphism $$\Phi\colon A\to\prod_n CL(\mathcal{H})/\bigoplus_n CL(\mathcal{H})$$ which is liftable to a cpc map $\varphi\colon A\to\prod_n CL(\mathcal{H})$.\\
(b) $A$ is  \emph{almost connective}, if there is a (not necessarily liftable) $*$-mono\-mor\-phism $\Phi\colon A\to\prod_n CL(\mathcal{H})/\bigoplus_n CL(\mathcal{H})$.
\end{definition}
For a discrete  group $G$, we define $I(G)$ to be the augmentation ideal, i.e.\ the kernel of the trivial representation $C^*(G) \to \C$.
We will sometimes say that a discrete amenable group $G$ is connective if the $C^*$-algebra $I(G)$ is connective. Note that (almost) connective $C^*$-algebras do not have nonzero projections. Thus any connective $C^*$-algebra is nonunital. Our definition allows that the zero $C^*$-algebra $\{0\}$ is connective.
	
Let $A$ and $B$ be $C^*$-algebras. An \emph{asymptotic morphism}  is a family of maps $\{\varphi_t \colon A \to B\}_{t \in [0,\infty)}$ such that
\begin{enumerate}[a)]
	\item for each $a \in A$ the map $t \mapsto \varphi_t(a)$ is norm-continuous and bounded,
	\item for all $a,b \in A$ and $\lambda \in \C$, we have
	\begin{align*}
		\lim_{t \to \infty} \lVert \varphi_t(a + \lambda\,b) - (\varphi_t(a) + \lambda\,\varphi_t(b))\rVert & = 0 \\	
		\lim_{t \to \infty} \lVert \varphi_t(ab) - \varphi_t(a)\,\varphi_t(b)\rVert & = 0 \\	
		\lim_{t \to \infty} \lVert \varphi_t(a^*) - \varphi_t(a)^*\rVert & = 0 \ .
	\end{align*}
\end{enumerate}
A \emph{discrete asymptotic morphism} $(\varphi_n)_{n \in \N}$ between $A$ and $B$ is a family of maps $\varphi_n \colon A \to B$ that satisfies the analogous conditions as a) and b) above with the index set replaced by $\N$. A homotopy between two (discrete) asymptotic morphisms $(\varphi^0_t)_{t \in I}$ and $(\varphi^1_t)_{t \in I}$ is a (discrete) asymptotic morphism $H_t \colon A \to C[0,1] \otimes B$, such that $\text{ev}_i \circ H_t = \varphi^i_t$ for all $t \in I$, where $I$ either denotes $[0,\infty)$ or $\N$. We will say that a (discrete) asymptotic  morphism $(\varphi_t)_{t \in I}$ is completely positive and contractive (cpc) if each of the maps $\varphi_t$ is cpc. The corresponding homotopy classes will be denoted as follows:
\begin{itemize}
	\item $[[A, B]]$ -- homotopy classes of asymptotic morphisms,
	\item $[[A, B]]_{\N}$ -- homotopy classes of discrete asymptotic morphisms,
	\item $[[A, B]]_{\N}^{cp}$ -- homotopy classes of discrete cpc  asymp.\ morphisms
\end{itemize}

\subsection{Characterizations of connectivity}
In the following we give two more characterizations of connectivity for exact and respectively nuclear $C^*$-algebras.

\begin{proposition}\label{products:exact}
  Let $A$ be a separable exact $C^*$-algebra. Then $A$ is connective if and only if there is an injective $*$-homomorphism
  $\pi\colon A \to \Cuntz{2}$ which is null-homotopic as a discrete cpc asymptotic morphism. This means that $[[\pi]]=0$ in the set $[[A,\Cuntz{2}]]_{\N}^{cp}$.
\end{proposition}

\begin{proof} $(\Rightarrow)$ 
By assumption, there is a cpc discrete asymptotic morphism $\{\varphi_n:A\to C[0,1]\otimes \Cuntz{2}\}_n$ such that
$\varphi_n^{(0)}=\pi$ is an injective $*$-homomorphism and $\varphi_n^{(1)}=0$. Thus, we can view $\{\varphi_n\}_n$ as an injective discrete asymptotic morphism $\{\varphi_n:A\to C_0[0,1)\otimes \Cuntz{2}\subset CL(\mathcal{H})\}_n$ and hence $A$ is connective.

$(\Leftarrow)$ Suppose that $A$ is a separable exact connective $C^*$-algebra.
By \cite[Prop.~2.11]{Dad-Pennig-homotopy-symm} it follows that $[[A,\Cuntz{2}\otimes \K]]_{\N}^{cp}$ is an abelian group. By Kirchberg's embedding theorem, there is an injective $*$-homomorphism $\pi\colon A \to \Cuntz{2}\otimes \K$.
Moreover $\pi \oplus \pi \colon A \to \Cuntz{2}\otimes \K$ is unitarily homotopy equivalent to $\pi$.
It follows that $[[\pi]]\oplus [[\pi]]=[[\pi]]$ in the group $[[A,\Cuntz{2}\otimes \K]]_{\N}^{cp}$
and hence $[[\pi]]=0$. After embedding $\Cuntz{2}\otimes \K$ into $\Cuntz{2}$ we obtain the desired conclusion.
\end{proof}

\begin{proposition}\label{products} Let $A$ be a separable nuclear $C^*$-algebra.
The following properties are equivalent.
\begin{enumerate}[(i)]
	\item $A$ is connective.
	\item $A\otimes B$ is connective for some $C^*$-algebra $B$ that contains a nonzero projection.
	\item $A\otimes B$ is connective for all $C^*$-algebras $B$
	\item $[[A,\Cuntz{2}\otimes \K]]=0$.
	\item $[[A,L(\mathcal{H})\otimes \K]]=0$.
\end{enumerate}
\end{proposition}

\begin{proof} The equivalences $(i) \Leftrightarrow (ii)  \Leftrightarrow  (iii)$ were established in \cite{Dad-Pennig-homotopy-symm}. (For $(ii) \Rightarrow (i)$ observe that $A$ is a subalgebra of $A \otimes B$ if $B$ contains a nonzero projection.)

$(i) \Rightarrow (iv)$ and $(i) \Rightarrow (v)$.
Let $B$ be a $\sigma$-unital $C^*$-algebra such that $KK(A,B)=0$, for instance $B=\Cuntz{2}$ or $B=L(\mathcal{H})$.
If $A$ is connective, then $A$ is homotopy symmetric and hence $[[A,B\otimes \K]]\cong KK(A,B)=0$ by \cite[Thm.~3.1]{Dad-Pennig-homotopy-symm}. Note that even though \cite[Thm.~3.1]{Dad-Pennig-homotopy-symm} was stated for separable $C^*$-algebras $B$ it is routine to extend the result to general $C^*$-algebras using the separability of $A$.

$(iv) \Rightarrow (i)$ and $(v) \Rightarrow (i)$.
Fix an embedding $\pi \colon A \to \Cuntz{2} \subset L(\mathcal{H})\otimes \K$ and regard it as a constant asymptotic morphism $\{\pi_t \colon A \to L(\mathcal{H})\otimes \K\}_t$.
By assumption, $[[\pi_t]]=0$ in $[[A,L(\mathcal{H})\otimes \K]]$ and hence by restriction,
$[[\pi_n]]=0$ in $[[A,L(\mathcal{H})\otimes \K]]_{\N}$.  We shall
view $L(\mathcal{H})\otimes \K$ as a subalgebra of $L(\mathcal{H})$.
The corresponding homotopy from the constant discrete morphism $\{\pi_n\}_n$ to zero
will induce an embedding $\Phi\colon A\to\prod_n CL(\mathcal{H})/\bigoplus_n CL(\mathcal{H})$ which is liftable to a cpc map $\varphi\colon A\to\prod_n CL(\mathcal{H})$ by the nuclearity of $A$.
\end{proof}

\subsection{Extensions of connective $C^*$-algebras}
Connectivity of $C^*$-algebras has a plethora of permanence properties as proven in \cite[Thm.\ 3.3]{Dad-Pennig-homotopy-symm}. In particular, it is inherited by split extensions \cite[Thm.~3.3~(d)]{Dad-Pennig-homotopy-symm}. In the following theorem this result is extended to non-split extensions as well.

\begin{theorem}\label{thm:extensions} Let $0\to J \to A \to B \to 0$ be an exact sequence of separable nuclear $C^*$-algebras.
If $J$ and $B$ are connective, then $A$ is connective.
\end{theorem}
\begin{proof} Since  connectivity passes to nuclear subalgebras we may replace the given extension by
\[0\to J\otimes \Cuntz{2} \otimes \K \to A \otimes \Cuntz{2} \otimes \K \to B\otimes \Cuntz{2} \otimes \K \to 0.\]
Adding to this extension a trivial absorbing extension, using the addition in $\mathrm{Ext}$-theory,
we obtain an absorbing extension
\begin{equation}\label{eq:split}0\to J\otimes \Cuntz{2} \otimes \K \to E \to B\otimes \Cuntz{2} \otimes \K \to 0,\end{equation}
which by construction has the property that $A \subset  A \otimes \Cuntz{2} \otimes \K \subset E $, see  for instance \cite[Lemma 2.2]{BroDad:qdext}.
Since $Ext(B\otimes \Cuntz{2}, J\otimes \Cuntz{2})=0$ as $\Cuntz{2}$ is KK-contractible and we are dealing with an absorbing extension, it follows that the extension ~\eqref{eq:split} splits by \cite[Sec.~7]{paper:KasparovKFunc} and so
$E$ is connective by \cite[Thm. 3.3]{Dad-Pennig-homotopy-symm}. We conclude that $A\subset E$ is connective.
\end{proof}
In the sequel we will need to use the following result from \cite{Dad-Pennig-homotopy-symm}, which is based on \cite{Blanchard:subtriviality}.
\begin{theorem}[\cite{Dad-Pennig-homotopy-symm}, Cor.~3.4.]\label{thm:fields}
 Let $A$ be a separable continuous field of nuclear $C^*$-algebras over a compact connected metrizable space $X$.
If one of the fibers of $A$ is connective, then $A$ is connective.
\end{theorem}

\begin{corollary} \label{cor:fields}
Let $A$ be a separable continuous field of nuclear $C^*$-algebras over a locally compact metrizable space $X$ that has no compact open subsets. Then $A$ is connective.
\end{corollary}

\begin{proof} \label{pf:fields}
Let $Y=X\cup \{y_0\}$ be the one-point compactification of $X$. Then $Y$ is a compact metrizable space which must be connected.
Indeed, arguing by contradiction, say that $Y=U\cup V$ with $U$, $V$ open and nonempty with  $y_0\in U$ and $U \cap V = \emptyset$. Then
$V=V\cap X$ is  both an open and compact subset of $X$.

We can view $A$ as a continuous field over $Y$ (see the remark on page~145 of \cite{Blanchard:subtriviality}) and the fiber over $y_0$ satisfies $A(y_0)=\{0\}$. It follows that $A$ is connective by Thm.~\ref{thm:fields}.
\end{proof}

\subsection{Obstructions to connectivity}
If $A$ is a $C^*$-algebra, we denote by $\widehat{A}$ the spectrum of $A$, which consists of all unitary equivalence classes of irreducible representations and by $\Prim{A}$ the primitive spectrum of $A$ consisting of kernels of irreducible representations. The unitary dual $\widehat{G}$ of a group $G$ identifies with $\widehat{C^*(G)}$. Recall that $\widehat{A}$ is topologized by pulling-back the Jacobson topology of $\Prim{A}$ under the natural map $\widehat{A}\to \Prim{A}$, $\pi\mapsto \mathrm{ker}\,{\pi}$, \cite{Dix:C*}.
 Let $\pi$ and $(\pi_n)_n$ be irreducible representations of $A$ acting on the same separable Hilbert space $H$. Suppose that
 $\|\pi_n(a)\xi-\pi(a)\xi\|\to 0$ for all $a \in A$ and $\xi\in H$. Then the sequence $(\pi_n)_n$ converges to $\pi$ in the topology of $\widehat{A}$, see \cite[Sec.~3.5]{Dix:C*}.
\begin{proposition}\label{prop:compact-open}
Let $A$ be a separable $C^*$-algebra.
\begin{enumerate}[(i)]
\item If $\mathrm{Prim}(A)$ contains a non-empty compact open subset, then $A$ is not connective.
\item If $A$ is nuclear and
$\mathrm{Prim}(A)$ is Hausdorff, then $A$ is connective if and only if $\mathrm{Prim}(A)$ does not contain a non-empty compact open subset.
\end{enumerate}
\end{proposition}
\begin{proof} (i) Set $X=\mathrm{Prim}(A)$. If $X$ has a non-empty compact open subset, then $A\otimes \Cuntz{2}$ contains a nonzero projection by \cite[Prop.~2.7]{Pasnicu-Rordam} and hence $A$ cannot be connective.

(ii) One implication follows from (i). For the other implication suppose that $X=\mathrm{Prim}(A)$ does not contain a non-empty compact open subset. Since $X$ is Hausdorff by assumption, $A$ is a nuclear separable continuous field over the locally compact space $X$, \cite{Fell}. This is explained in detail in \cite[Sec.\ 2.2.2]{BK}. Now we apply Cor.\ \ref{cor:fields}.
\end{proof}

We would like thank Gabor Szabo for pointing out the following invariance property of connectivity.
\begin{proposition} Let $A$ and $B$ be separable nuclear $C^*$-algebra with homeomorphic primitive spectra.
Then $A$ is connective if and only if $B$ is connective.
\end{proposition}
\begin{proof} Kirchberg's classification theorem \cite{Kir:Michael} implies that if $A$ and $B$ are as in the statement, then
$A\otimes \Cuntz{2} \otimes \K \cong B\otimes \Cuntz{2} \otimes \K$.  The desired conclusion follows now from Proposition~\ref{products}.
\end{proof}

\begin{definition} \label{def:shielded}
	Let $A$ be a separable $C^*$-algebra. A point $\pi \in \widehat{A}$ is called \emph{shielded}, if $\widehat{A} \setminus \{\pi\} \neq \emptyset$ and any sequence $(\pi_n)_n$ in $\widehat{A}\,\setminus \{\pi\}$ which converges to $\pi$ has a subsequence that converges to another point $\eta \in \widehat{A}\setminus \{\pi\}$.
\end{definition}

\begin{lemma} \label{lem:shielded}
Let $A$ be a unital separable $C^*$-algebra. If a point $\pi \in \widehat{A}$ is closed and shielded, then
$I= \ker\,\pi$ is not connective.
\end{lemma}
\begin{proof} Observe that $I\neq \{0\}$, since $\widehat{A}\setminus \{\pi\}\neq \emptyset$ and $\{\pi\}$ is closed.
By Proposition~\ref{prop:compact-open} it suffices to show that $\text{Prim}(I)=\text{Prim}(A)\setminus \{I\}$  is a nonempty compact-open subset of
$\text{Prim}(A)$.
Since $\{\pi\}$ is closed, it follows that $q^{-1}(I)=\{\pi\}$ and hence $q(\widehat{A}\setminus \{\pi\})=\text{Prim}(A)\setminus \{I\}$. The quotient map $q \colon \widehat{A} \to \Prim{A}$ is continuous and open, since the topology of $\widehat{A}$ is defined as the preimage of the topology of $\Prim{A}$.
Therefore, the lemma follows if we show that $\widehat{A}\setminus \{\pi\}$ is  compact and open.  $\widehat{A}\setminus \{\pi\}$ is open because $\{\pi\}$ is closed. Since $\widehat{A}$ is compact and satisfies the second axiom of countability \cite{Dix:C*}, it suffices to show that $\widehat{A}\setminus \{\pi\}$  is sequentially compact, \cite[p.~138]{book:Kelley}. Let $(\pi_n)_n$ be a sequence in $\widehat{A}\setminus \{\pi\}$. By compactness of $\widehat{A}$ it contains a subsequence $(\pi_{n_k})_k$ converging in $\widehat{A}$. If it converges to $\pi \in \widehat{A}$, then it has a subsequence that converges to some other point $\eta \in \widehat{A}\setminus \{\pi\}$, because $\pi$ is shielded.
Hence $\widehat{A}\setminus \{\pi\}$ is also compact.
\end{proof}

\begin{corollary} \label{cor:iota_shielded}
Let $G$ be a countable discrete group. If the trivial representation $\iota \in \widehat{G}$ is shielded, then $I(G)$ is not connective. 
\end{corollary}
\begin{proof}
  Since $\iota$ is a one-dimensional representation, it follows  that $\{\iota\}$ is closed  in~$\widehat{G}$.
  Thus, the statement follows from Lemma~\ref{lem:shielded}.
\end{proof}

\section{Connectivity of crystallographic groups}
It is known that there are precisely $10$ closed flat $3$-dimensional manifolds.
Conway and Rossetti \cite{Conway-Rossetti} call these manifolds
 platycosms (``flat universes"). The Hantzsche-Wendt manifold \cite{Hantzsche-Wendt}, or the didicosm in the terminology of \cite{Conway-Rossetti},   is the only platycosm with finite homology.
 Its fundamental group $G$, known as the Hantzsche-Wendt group,  is generated by two elements $x$ and $y$ subject to two relations:
\[x^2yx^2=y,\quad y^2xy^2=x.\]
The group $G$ is one of the classic torsion free 3-dimensional crystallographic groups,  \cite{Hantzsche-Wendt, Conway-Rossetti}. It is useful to introduce the notation $z=(xy)^\inv$.

A  concrete realization of $G$ as rigid motions  of $\R^3$ is given by the following transformations $X$, $Y$, $Z$ that correspond to the group elements $x$, $y$ and $z$.
\[X (\xi)= A \xi+{a}, \quad Y (\xi)= B \xi+{b},\quad Z (\xi)= C \xi+{c},\quad \xi\in \R^3,\]
where
\[A=\begin{pmatrix}
       1 & 0 & 0  \\
       0 & -1 & 0  \\
       0 & 0 & -1
     \end{pmatrix},
     \quad B=\begin{pmatrix}
       -1 & 0 & 0  \\
       0 & 1 & 0  \\
       0 & 0 & -1
     \end{pmatrix},\quad
     C=\begin{pmatrix}
       -1 & 0 & 0  \\
       0 & -1 & 0  \\
       0 & 0 & 1
     \end{pmatrix},\] and
\[{a}=\begin{pmatrix}
       1/2   \\
       1/2   \\
        0
     \end{pmatrix}, \quad {b}=\begin{pmatrix}
       0  \\
       1/2   \\
       1/2
     \end{pmatrix},\quad {c}=\begin{pmatrix}
       1/2   \\
       0  \\
        1/2
     \end{pmatrix}.\]

The transformations $X^2$, $Y^2$ and $Z^2$ are just translations by unit vectors in the positive directions of the coordinate axes.

One shows (independently of the previous concrete realization) that the elements $x^2$, $y^2$ and $z^2$ commute. Moreover one has the following relations in $G$:
\[x x^2x^{-1}=x^2, \quad x y^2x^{-1}=y^{-2},\quad x z^2x^{-1}=z^{-2}\]
\[y x^2y^{-1}=x^{-2}, \quad y y^2y^{-1}=y^{2},\quad y z^2y^{-1}=z^{-2}\]
\[z x^2z^{-1}=x^{-2}, \quad z y^2z^{-1}=y^{-2},\quad z z^2z^{-1}=z^{2}\]

The subgroup $N$ of $G$ generated by $x^2,y^2$ and $z^2$ is normal in $G$ and it is isomorphic to $\Z^3\cong \Z x^2\oplus \Z y^2\oplus \Z z^2$. Let $q:G \to H=G/N$ denote the quotient map.
\[
\xymatrix{
	1 \ar[r] & N \ar[r] & G \ar[r]^-{q} & H \ar[r] & 1
}
\]
$H$ is isomorphic to $\Z/2\oplus \Z/2$ with generators are  $q(x)$ and $q(y)$.

For later use, we will need the following identities that hold in $G$.
\begin{gather}\label{eq:useful}
 x^\inv y=yxy^2z^{-2}=zx^2z^{-2},\quad x^\inv z=yz^{2},\\  \notag y^\inv x=zx^2,\quad  y^\inv z=x (x^{-2}y^2).
\end{gather}

\subsection{Induced representations and the unitary dual of $G$} \label{sec:spectrum-HW}
Based on Corollary~\ref{cor:iota_shielded} we will show that $I(G)$ for the Hantzsche-Wendt group $G$ is not connective. This requires a thorough analysis of the spectrum $\widehat{G}$.

Our basic reference for this section is the book of Kaniuth and Taylor \cite{book:Kaniuth-Taylor}.
The  unitary dual of $G$ consists of unitary equivalence classes of irreducible unitary representations of $G$ and is denoted by $\widehat{G}$.
$G$ acts on $\widehat{N}\cong \mathbb{T}^3$ by $g\cdot \chi = \chi(g\cdot g^{-1})$.
If we identify the character $\chi\in \widehat{N}$ with the point $(\chi(x^2),\chi(y^2),\chi(z^2))=(u,v,w)\in \T^3$, then the action of $G$ is described as follows:
\[x\cdot (u,v,w)=(u,\bar{v},\bar{w}),\quad y\cdot (u,v,w)=(\bar{u},v,\bar{w}), \quad z\cdot (u,v,w)=(\bar{u},\bar{v},w).\]
The stabilizer of a character $\chi$ is the subgroup $G_\chi$ of $G$ defined by
$G_\chi=\{g\in G\colon \chi(g\cdot g^{-1})=\chi(\cdot)\}$. It is clear that $N\subset G_\chi$  and that there is a bijection from $G/G_\chi$ onto the orbit of $\chi$. In particular, the orbits of the action of $G$ on $\widehat{N}$ can only have length $1$, $2$ or $4$.
Mackey has shown that each irreducible representation $\pi\in \widehat{G}$ is supported by the orbit of some character $\chi\in \widehat{N}$,  in the sense that the restriction of $\pi$ to $N$ is unitarily equivalent to some multiple $m_\pi$ of the direct sum of the characters in the  orbit of $\chi$.
\[\pi_{|_N }\sim m_\pi \bigoplus_{g\in G/G_\chi} \chi(g\cdot g^{-1}). \]
In the sum above $g$ runs through a set of coset representatives.

Mackey's theory has a particularly nice form for  virtually abelian discrete groups.
Let $\Omega\subset \widehat{N}$ be a subset which intersects each orbit of $G$ exactly once.
For each $\chi\in \widehat{N}$, let  $\widehat{G_\chi}$ be the unitary dual of the stabilizer group $G_\chi$
and denote by $\widehat{G_\chi}^{(\chi)}$ the subset of $\widehat{G_\chi}$ consisting of classes of irreducible representations $\sigma$ of $G_\chi$ such the restriction of $\sigma$ to $N$ is unitarily equivalent to a multiple of $\chi$.
Then, according to \cite [Thm. 4.28]{book:Kaniuth-Taylor}
\begin{theorem}\label{thm:Mackey} $\widehat{G}=\left\{\mathrm{ind}_{G_\chi}^{\,G}(\sigma)\colon \sigma \in \widehat{G_\chi}^{(\chi)},\,\,\chi\in \Omega\right\}.$
\end{theorem}

Let $\iota$ be the trivial representation of $G$. We will prove that $\iota \in \widehat{G}$ is shielded by showing that any sequence $(\pi_n)_n$ of points in $\widehat{G}\setminus \{\iota\}$ that converges to $\iota$ has a subsequence which is convergent to a point $\eta\neq \iota$.

Let $R_\ell\subset \widehat{G}$ consist of those classes of irreducible representations which lie over $\ell$-orbits, i.e.\ the orbits of length $\ell$. Write $\widehat{G}$ as the disjoint union $\widehat{G}=R_1\cup R_2 \cup R_4$. It suffices to assume that all the elements $\pi_n$ belong to the same subset $R_\ell$. We distinguish the three possible cases for $\ell$:

\vspace{2mm}
\paragraph{\textbf{1-orbits}}
Consider the characters of $N$ of the form $\chi=(\ep_1,\ep_2,\ep_3)$, where $\ep_1,\ep_2,\ep_3 \in\{\pm 1\}$.
These are precisely the points in $\widehat{N}$ which are fixed under the action of $G$.
In other words $G_\chi=G$. Let  $(\pi_n)_n$ be a sequence of elements in $R_1 \subset \widehat{G}$ and such that $(\pi_n)_n$ is convergent to $\iota$. Since the restriction of $\pi_n$ to $N$ is a multiple of a character $\chi_n=(\ep_1(n),\ep_2(n),\ep_3(n))$, it follows that $\ep_1(n)=\ep_2(n)=\ep_3(n)=1$ for all sufficiently large $n$ and hence
since $\pi_n$ is irreducible, there is $m$ such that $\pi_n=\iota$ for $n\geq m$. Hence there is no sequence in $R_1 \setminus \{\iota\}$ which converges to $\iota$.

\vspace{2mm}
\paragraph{\textbf{2-orbits}}
The characters $\chi\in \widehat{N}$ with orbits of length two are those $\chi=(u,v,w)$ where precisely
only one of the coordinates is not equal to $\pm 1$. Let us argue first that if  $(\pi_n)_n$ is a sequence of elements in $R_2 \subset \widehat{G}$ such that $\pi_n$ lies over the orbit of $\chi_n=(u_n,v_n,w_n)$ and $(\pi_n)_n$ is convergent to $\iota$, then two of the coordinates $u_n,v_n,w_n$ must be equal to $1$ for all sufficiently large~$n$.

 Suppose that each $\pi_n$ lies over the orbit of a character $\chi_n$ of the form $\chi_n=(u_n,\ep_2(n),\ep_3(n))$ where $u_n\neq \pm 1$ and
  $\ep_2(n),\ep_3(n)\in \{\pm 1\}$. Then $G_{\chi_n}$ is generated by $x,y^2,z^2$ and $\{e,y\}$ are coset representatives for $G/G_{\chi_n}$.

   Since
 $\left.\pi_n\right|_{N}\sim m_n (\chi_n(\cdot)\oplus \chi_n(y\cdot y^{-1}))$, it follows that
\[\pi_n(y^2)\sim m_n\begin{pmatrix} \ep_2(n) & 0\\ 0 & \ep_2(n)\end{pmatrix},\quad
\pi_n(z^2)\sim m_n\begin{pmatrix} \ep_3(n) & 0\\ 0 & \ep_3(n)
\end{pmatrix}\]
and hence if $(\pi_n)_n$ converges to $\iota$, then we must have $\ep_2(n)=\ep_3(n)=1$ for all sufficiently large $n$.
The cases $\chi_n=(\ep_1,v_n,\ep_3)$ and $\chi_n=(\ep_1,\ep_2,w_n)$ are treated similarly.

\vskip 12pt
In view of the discussion above, it suffices to focus
 on characters of $N$ the form $\chi=(u,1,1)$. The orbit of $\chi$ consists of two points,
 $(u,1,1)$ and  $(\bar{u},1,1)$.
 The corresponding stabilizer $G_\chi$
 is generated by $x,y^2$ and $z^2$. In particular $G_\chi=N\cup xN$ and $G=G_\chi \cup y \,G_\chi$.
 The exact sequence
 \[
 	\xymatrix{
 		1 \ar[r] & N \ar[r] & G_{\chi} \ar[r] & \Z/2\Z \ar[r] & 1
 	}
 \]
 does not split since $G$ is torsion free.
 The quotient $G/G_\chi$ is generated by the coset of $y$.
 Let $\sigma \in \widehat{G}_\chi$ be an irreducible representation of $G_\chi$ whose restriction to $N$ is a multiple of $\chi$. Since $\chi(y^2)=\chi(z^2)=1$, it follows that $\sigma$ factors through $G_{\chi}/\Z^2$.
 Moreover we have a nontrivial central extension
 \[
 	\xymatrix{
	 	1 \ar[r] & \Z \ar[r] & G_{\chi}/\Z^2 \ar[r] & \Z/2\Z \ar[r] & 1
	 }
 \]
 where the normal subgroup is generated by the image of $x^2$ under the map $N \to N /\langle y^2,z^2\rangle $ and the quotient group is generated by the image of $q(x)$ under the map $H \to H/\langle q(y)\rangle $.
 Since $G_{\chi}/\Z^2$ is an abelian group, $\sigma$ must be a character such that $\sigma(x)^2=\sigma(x^2)=\chi(x^2)=u$. Thus $\sigma(x)=a\in \T$ with $a^2=u$.
 Let us compute the representation $\pi=\mathrm{ind}_{G_\chi}^{\,G}(\sigma)$ of $G$ induced by $\sigma$.
 It acts on the Hilbert space \[H_\pi=\{\xi:G \to \C\colon \xi(gh)=\sigma(h^{-1})\xi(g),\quad g \in G,h\in G_{\chi}\}.\]  Since $G=G_\chi \cup y \,G_\chi$,
 we can identify $H_\pi$ with $\C^2$ via the isometry $\xi\mapsto (\xi(e),\xi(y))$.
 Then $\pi(g)\xi=\xi(g^{-1}\cdot)$ can be described using (\ref{eq:useful}) as follows:
\begin{equation*}
	\begin{aligned}
 \pi(x)\xi(e)&=\xi(x^{-1})=\sigma(x)\xi(e)=a\xi(e) \\
 \pi(x)\xi(y)&=\xi(x^{-1}y)=\xi (yx y^2z^{-2})=\sigma( z^{2} y^{-2}x^\inv)\xi(y)=\bar{a}\xi(y) \\
 \pi(y)\xi(e)&=\xi(y^{-1})=\xi (y\cdot y^{-2})=\sigma(y^2)\xi(y)=\xi(y) \\
  \pi(y)\xi(y)&=\xi(y^\inv \cdot y)=\xi(e) \\		
	\end{aligned}	
\end{equation*}
which produces the following matrices with respect to the basis given above:
  \begin{equation}\label{eq:2}
  \pi(x)=\begin{pmatrix}
             a & 0 \\
             0 & \bar{a}
           \end{pmatrix}, \quad \pi(y)=\begin{pmatrix}
             0 & 1 \\
             1 & 0
           \end{pmatrix}, \quad \pi(z)=\begin{pmatrix}
             0 & a \\
             \bar{a} & 0
           \end{pmatrix}
    \end{equation}
    Corresponding to the characters $(1,v,1)$ and $(1,1,w)$ we obtain the irreducible representations, where we use the isometries $\xi \mapsto (\xi(e), \xi(x))$ and $\xi \mapsto (\xi(e), \xi(y))$ respectively:
    \begin{equation}\label{eq:21}
   \quad \pi(x)=\begin{pmatrix}
             0 & 1 \\
             1 & 0
           \end{pmatrix}, \quad \pi(y)=\begin{pmatrix}
             b & 0 \\
             0 & \bar{b}
           \end{pmatrix}, \quad \pi(z)=\begin{pmatrix}
             0 & \bar{b} \\
            b & 0
           \end{pmatrix},\quad b^2=v,
    \end{equation}        and \begin{equation}\label{eq:22}
  \pi(x)=\begin{pmatrix}
              0 & \bar{c} \\
             c & 0
           \end{pmatrix}, \quad \pi(y)=\begin{pmatrix}
             0 & 1 \\
             1 & 0
           \end{pmatrix}, \quad \pi(z)=\begin{pmatrix}
             c & 0\\
             0 &\bar{c}
           \end{pmatrix},\quad c^2=w.
    \end{equation}

Let $(\pi_{n})_n$ be a sequence in $R_2$ that converges to $\iota$ in $\widehat{G}$. Arguing by symmetry, we  may assume that each $\pi_{n}$ is given by the formulas \eqref{eq:2} corresponding to a sequence of points $u_n \in \T$ with $u_n \notin \{\pm 1\}$. Since $\pi_{n} \to \iota$ it follows from the equation~\eqref{eq:2} that $u_n \to 1$. Again from \eqref{eq:2} we can compute the limits of the sequences $\pi_{n}(x)$ and $\pi_{n}(y)$ in $U(2)$. This gives the representation $\pi \colon G \to U(2)$:
 \begin{equation*}
  \pi(x)=\begin{pmatrix}
             1 & 0 \\
             0 & 1
           \end{pmatrix}, \quad \pi(y)=\begin{pmatrix}
             0 & 1 \\
             1 & 0
           \end{pmatrix}
    \end{equation*}

  It is clear that $\pi$ is a representation of $G$ that factors through the left regular representation of $\Z/2$. Decompose $\pi$ into a direct sum of characters $\pi\sim \iota \oplus \eta$. Then $\eta$ is not equivalent to $\iota$ and $\pi_{n}\to \eta$ in $\widehat{G}$.
\vspace{2mm}
\paragraph{\textbf{4-orbits}}
  Let $\chi=(u,v,w)\in \T^3$ be a character of $N$ with $u,v,w \notin\{\pm 1\}$.
  Its orbit under the action of $G$ consists of four points and $G_\chi=N$.
 Let us compute the representation $\pi=\mathrm{ind}_{N}^{\,G}(\chi)$ of $G$ induced by $\sigma$.
 It acts on the Hilbert space $H_\pi=\{\xi:G \to \C\colon \xi(gh)=\chi(h^{-1})\xi(g),\quad g \in G,h\in N\}$. Thus one can identify $H_\pi$ with $\C^4$ via the isometry $\xi\mapsto (\xi(e),\xi(x),\xi(y),\xi(z))$.
 Using the identities \eqref{eq:useful}, we verify that $\pi(g)\xi=\xi(g^{-1}\cdot)$ is described as follows:
\begin{equation*}
	\begin{aligned}
  		\pi(x)\xi(e)&=\xi(x^\inv)=\xi(x\cdot x^{-2})=\chi(x^2)\xi(x)=u\xi(x) \\
  		\pi(x)\xi(x)&=\xi(e) \\
    	\pi(x)\xi(y)&=\xi(x^\inv y)=\xi (z x^2z^{-2})=\chi(z^2 x^{-2}  )\xi(z)=w\bar{u}\xi(z) \\
		\pi(x)\xi(z)&=\xi(x^\inv z)=\xi (y z^{2})=\chi(z^{-2})\xi(y)=\bar{w}\xi(y)	\\[4mm]	
		\pi(y)\xi(e)&=\xi(y^\inv)=\xi(y \cdot y^{-2})=\chi(y^2)\xi(y)=v \xi(y) \\
 		\pi(y)\xi(x)&=\xi(y^\inv x)=\xi(zx^{2})=\chi(x^{-2})\xi(z)=\bar{u}\xi(z) \\
		\pi(y)\xi(y)&=\xi(e) \\
		\pi(y)\xi(z)&=\xi(y^\inv z)=\xi (x \cdot x^{-2}y^2)=\chi(y^{-2}x^2)\xi(x)=\bar{v}u\xi(x)
	\end{aligned}
\end{equation*}
producing the matrices:
     \begin{equation}\label{eq:4}
     \pi(x)=\begin{psmallmatrix}
       0 & u & 0 & 0 \\[1mm]
       1 & 0 & 0 & 0 \\[1mm]
       0 & 0 & 0 & \bar{u}w \\[1mm]
       0& 0 & \bar{w} & 0
     \end{psmallmatrix},\quad \pi(y)=
     \begin{psmallmatrix}
       0 & 0 & v & 0 \\[1mm]
       0 & 0 & 0 & \bar{u} \\[1mm]
       1 & 0 & 0 & 0 \\[1mm]
       0& u\bar{v} & 0 & 0
     \end{psmallmatrix}
      \end{equation}
     It follows that
     \begin{gather}\label{eq:41}\pi(x^2)=\begin{psmallmatrix}
       u & 0 & 0 & 0 \\[1mm]
       0 & u & 0 & 0 \\[1mm]
       0 & 0 & \bar{u} & 0 \\[1mm]
       0& 0 & 0 & \bar{u}
     \end{psmallmatrix},\quad \pi(y^2)=\begin{psmallmatrix}
       v & 0 & 0 & 0 \\[1mm]
       0 & \bar{v} & 0 & 0 \\[1mm]
       0 & 0 & v & 0 \\[1mm]
       0& 0 & 0 & \bar{v}
     \end{psmallmatrix},  \quad 
     \pi(z^2)=\begin{psmallmatrix}
       w & 0 & 0 & 0 \\[1mm]
       0 & \bar{w} & 0 & 0 \\[1mm]
       0 & 0 & \bar{w}& 0 \\[1mm]
       0& 0 & 0 & w
     \end{psmallmatrix}
     \end{gather}


Let $(\pi_n)_n$ be a sequence in $R_4$ converging to $\iota$ in $\widehat{G}$. Each $\pi_{n}$ is given by the formulas \eqref{eq:4} corresponding to a sequence of points $(u_k,v_k,w_k)\in \T^3$ with $u_k,v_k,w_k \notin\{\pm 1\}$. Since $\pi_{n}\to \iota$ it follows from equation~\eqref{eq:41} that $u_k,v_k,w_k \to 1$. Again from \eqref{eq:4} we can compute the limits of the sequences $\pi_{n}(x)$ and $\pi_{n}(y)$ in the space of unitary operators $U(4)$. This gives the representation $\pi \colon G \to U(4)$:
\begin{equation*}
     \pi(x)=
     \begin{psmallmatrix}
       0 & 1 & 0 & 0 \\[1mm]
       1 & 0 & 0 & 0 \\[1mm]
       0 & 0 & 0 & 1\\[1mm]
       0& 0 & 1 & 0
     \end{psmallmatrix},\quad
     \pi(y)=
     \begin{psmallmatrix}
       0 & 0 & 1 & 0 \\[1mm]
       0 & 0 & 0 & 1 \\[1mm]
       1 & 0 & 0 & 0 \\[1mm]
       0& 1 & 0 & 0
     \end{psmallmatrix}
\end{equation*}
It is clear that $\pi$ is a representation of $G$ that factors through $H=\Z/2 \times \Z/2$. Decompose $\pi$ into a direct sum of characters $\eta_i$. Since $\pi$ is not equivalent to a multiple of the trivial representation, it follows that at least one of these characters is not equivalent to $\iota$.  On the other hand $\pi_{n} \to \eta_i$ in $\widehat{G}$ for all $i$. Combining the above analysis with Corollary~\ref{cor:iota_shielded} we obtain immediately
\begin{corollary}\label{cor:HW_not_connective}
	If $G$ is the Hantzsche-Wendt group, then $I(G)$ is not connective.
\end{corollary}

It was conjectured in \cite{Dadarlat-almost} that if $G$ is a torsion free discrete amenable group, then
$[[I(G),\K]]\cong KK(I(G),\C)$. We argue now that this conjecture fails for the Hantzsche-Wendt group.
Indeed this follows from the previous corollary in conjunction with the following lemma.
\begin{lemma}\label{lemma:reduction}
 Let $G$ be a residually finite torsion free discrete amenable group which admits a  classifying space with finitely generated K-homology group $K_1(BG)$.
Then $[[I(G),\K]]\cong KK(I(G),\C)$ if and only if $I(G)$ is connective.\end{lemma}
\begin{proof} Suppose first that $[[I(G),\K]]\cong KK(I(G),\C)$. Since $G$ is amenable and residually finite it follows that $C^*(G)$ is residually finite dimensional.
Since $G$ is amenable, $G$ satisfies the Baum-Connes conjecture and $C^*(G)$ satisfies the UCT by results of Higson and Kasparov \cite{Higson-Kasparov01} and Tu \cite{Tu-BC}.
In particular we have a short exact sequence
\[0 \to Ext^1(K_1(C^*(G)),\Z)\to KK(C^*(G),\C)\to \mathrm{Hom}(K_0((C^*(G)),\Z)\to 0\]
 Let $\pi_n:C^*(G)\to M_{d(n)}(\C)$ be a separating sequence of finite dimensional representations.
 The restriction of $\pi_n$ to $I(G)$ will be denoted by~$\sigma_n$. By \cite[Prop.~3.2]{Dadarlat-almost}
 $(\pi_n)_*=d(n)\iota_*: K_0(C^*(G))\to \Z$ and hence $[\sigma_n]\in Ext^1(K_1(I(G)),\Z)\subset KK(I(G),\C)$ is a torsion element since $K_1(I(G))\cong K_1(BG)$ is finitely generated. After replacing $\pi_n$ by a suitable multiple of itself we have arranged that $[\sigma_n]=0$ in $KK(I(G),\C)$ and hence $[[\sigma_n]]=0$ in $[[I(G),\K]]$.
 Since the sequence $(\sigma_n)$ separates the elements of $I(G)$ it follows that $I(G)$ is connective.

 The converse is contained in the main result of \cite{Dad-Pennig-homotopy-symm} which shows that if $A$ is a separable nuclear connective $C^*$-algebra, then $[[A,\K]]\cong KK(A,\C)$.
\end{proof}

\subsection{Crystallographic groups with cyclic holonomy}
In this section we are going to show that torsion free crystallographic groups with cyclic holonomy are connective. Apart from this we isolate a lemma, which proves that $I(G)$ for a group $G$ which is a finite extension of a connective group always contains a ``big'' connective ideal. In particular, the lemma also holds for the Hantzsche-Wendt group.

The proof of both results uses some tools from the index theory of $C^*$-subalgebras. A reference is \cite{book:WatataniIndex}.
Let $\Gamma$ and $G$ be discrete groups and let $H$ be a finite group. Suppose that they fit into an exact sequence
\[
	\xymatrix{
		1 \ar[r] & \Gamma \ar[r] & G \ar[r]^-{q} & H \ar[r] & 1\ .
	}
\]
Let $E \colon C^*(G) \to C^*(\Gamma)$ be the faithful conditional expectation \cite[Ex.~1.2.3]{book:WatataniIndex} given on group elements by
\[
	E(g) = \begin{cases}
		g & \text{if } g \in \Gamma \\
		0 & \text{else}
	\end{cases} \ .
\]
Choose a lift $g_h \in G$ for each $h \in H$. The pairs $(g_h^{-1}, g_h)$ form a quasi-basis in the sense of \cite[Def.~1.2.2]{book:WatataniIndex}.
Let $\mathcal{E} = C^*(G)$ considered as a right Hilbert $C^*(\Gamma)$-module, where the right action is induced by the inclusion $C^*(\Gamma) \to C^*(G)$ and the inner product is given by $\langle a,b \rangle = E(a^*b)$ \cite[Sec.~2.1]{book:WatataniIndex}. Note that $\mathcal{E}$ is complete \cite[Prop.~2.1.5]{book:WatataniIndex}. The quasi-basis induces an isometric isomorphism of right Hilbert $C^*(\Gamma)$-modules $u \colon \mathcal{E} \to \ell^2(H) \otimes C^*(\Gamma)$ with
\[
	u(a) = \sum_h \delta_h \otimes E(g_ha)
\]
and inverse $u^* \colon \ell^2(H) \otimes C^*(\Gamma) \to \mathcal{E}$ with $u^*(\delta_h \otimes b) = g_h^{-1}\,b$. Let $\mathcal{L}_{C^*(\Gamma)}(\mathcal{E})$ be the bounded adjointable operators on $\mathcal{E}$ and denote by $\mathcal{K}_{C^*(\Gamma)}(\mathcal{E})$ the compact ones. Then we have $\mathcal{L}_{C^*(\Gamma)}(\mathcal{E}) \cong \mathcal{K}_{C^*(\Gamma)}(\mathcal{E}) \cong \K(\ell^2(H)) \otimes C^*(\Gamma)$. The left multiplication of $C^*(G)$ on $\mathcal{E}$ induces a $*$-homomorphism
\[
	\psi \colon C^*(G) \to \K(\ell^2(H)) \otimes C^*(\Gamma)
\]
with matrix entries $\psi_{h',h}(a) = E(g_{h'}\,a\,g_h^{-1})$. Suppose we have $a \in C^*(G)$ with $\psi(a) = 0$. Then
\[
	a = \frac{1}{\lvert H \rvert}\sum_{h,h'} g_{h'}^{-1}E(g_{h'}\,a\,g_h^{-1})g_h = 0\ .
\]
Hence, $\psi$ is injective.

\begin{lemma}
Let $\Gamma$ be a connective group and let $H$ be a finite group. Suppose that the group $G$ fits into a short exact sequence of the form
\[
	\xymatrix{
		1 \ar[r] & \Gamma \ar[r] & G \ar[r]^-{q} & H \ar[r] & 1\ .
	}
\]
Then $I(G,H) = \ker(I(q) \colon I(G) \to I(H))$ is connective as well.	
\end{lemma}

\begin{proof}
Let $\iota \colon C^*(\Gamma) \to \C$ be the trivial representation and let $\psi$ be the injective $*$-homomorphism constructed above. For all $b \in C^*(\Gamma) \subset C^*(G)$ we have $\psi_{h',h}(b) = \delta_{h',h}\,g_h\,b\,g_h^{-1}$. In particular, $\psi$ embeds the ideal $J$ generated by $I(\Gamma)$ into $\ker (\text{id} \otimes \iota) = \K(\ell^2(H)) \otimes I(\Gamma)$, which is connective. Hence $J$ is connective as well. It is clear that $J \subseteq I(G,H)$. Let $x \in I(G,H)$. By the property of the quasi-basis
\[
	0 = q(x) = \sum_{h\in H} q(E(x\,g_h^{-1}))\,h \quad \Rightarrow \quad q(E(x\,g_h^{-1})) = 0 \quad \forall h \in H \ .
\]
Since $I(G,H) \cap C^*(\Gamma) = I(\Gamma)$ we obtain that $E(x\,g_h^{-1}) \in I(\Gamma)$ for all $h \in H$ and therefore $x = \sum_{h} E(x\,g_h^{-1}))\,g_h \in J$, hence $J = I(G,H)$.
\end{proof}

The proof of the second result uses an induction over the rank of the free abelian subgroup based on the following observation.

\begin{lemma} \label{lem:fin_ext}
Let $m > 1$ and let $\Gamma$ and $G$ be countable discrete groups that fit into the following short exact sequence
\[
	\xymatrix{
		1 \ar[r] & \Gamma \ar[r] & G \ar[r]^-{\pi} & \Z/m\Z \ar[r] & 1\ .
	}
\]
Suppose that $\Gamma$ is connective and there are group homomorphisms $\varphi \colon G \to \Z$ and $q \colon \Z \to \Z/m\Z$, such that $\pi = q \circ \varphi$. Then $G$ is connective as well.
\end{lemma}

\begin{proof}
Let $\psi \colon C^*(G) \to \K(\ell^2(H)) \otimes C^*(\Gamma)$ be the injective $*$-homomorphism constructed above and let $\iota \colon C^*(\Gamma) \to \C$ be the trivial representation. Observe that $\rho =(\text{id} \otimes \iota) \circ \psi$ satisfies
\begin{align*}
	\rho(g\,\gamma)_{h',h} &= \iota(E(g_{h'}g\,\gamma\,g_h^{-1})) = \iota(E(g_{h'}g\,g_h^{-1})\,g_h\,\gamma\,g_h^{-1}) = \iota(E(g_{h'}g\,g_h^{-1}))
\end{align*}
for all $g \in G$ and $\gamma \in \Gamma$. In particular, $\rho$ factors through the $*$-homomorphism $C^*(G) \to C^*(\Z/m\Z)$ induced by $\pi$. By assumption this decomposes as
\[
\xymatrix{
	C^*(G) \ar[r]^-{\varphi} & C^*(\Z) \ar[r]^-{q} & C^*(\Z/m\Z)\ .
}
\]
Altogether we obtain that $\rho$ decomposes into a direct sum of one-dimensional representations, each of which is homotopic through representations to the trivial one. Hence, to show that $G$ is connective, it suffices to construct a path through discrete asymptotic morphisms connecting a faithful morphism with a direct sum of copies of $\rho$.

Choose a path witnessing the connectivity of $\Gamma$, i.e.\ a discrete asymptotic morphism
\[
	H_n \colon C^*(\Gamma) \to C([0,1]) \otimes M_{n}(\C)
\]
such that for $H_n^{(t)} = \text{ev}_t \circ H_n \colon C^*(\Gamma) \to M_n(\C)$ we have that $H_n^{(0)}$ is faithful and $H_n^{(1)}$ is a multiple of $\iota$. Then $(\text{id}_{M_m(\C)} \otimes H_n) \circ \psi$ has the desired properties. Hence, $G$ is connective.
\end{proof}

We need the following elementary fact:

\begin{lemma} \label{lem:gen_lift}
Let $a, b > 1$ be integers and consider the exact sequence
\[
\xymatrix{
	0 \ar[r] & \Z/a\Z \ar[r]^{\cdot b} & \Z/ab\Z \ar[r]^-{\pi} & \Z/b\Z \ar[r] & 0
}
\]
with $\pi(x) = x \mod b$. Any generator of $\Z/b\Z$ lifts to a generator of $\Z/ab\Z$.
\end{lemma}

\begin{proof}
Let $\bar{y} \in \Z/b\Z$ be a generator and let $y \in \{0,\dots, b-1\}$ be a representative. Let $p_1, \dots, p_s$ be the disctinct prime factors of $ab$, such that $p_1, \dots, p_r$ for $r \leq s$ are the ones not dividing $y$ and $p_{r+1}, \dots, p_s$ divide $y$. Since $\gcd(y,b) = 1$, the primes $p_{r+1}, \dots, p_s$ do not divide $b$. Let $x = y + p_1\dots p_r\, b$. We have for $i \in \{1, \dots, r\}$ and $j \in \{r+1, \dots, s\}$
\begin{align*}
	x &\equiv y \not\equiv 0 \mod p_i \ ,\\
	x &\equiv p_1\dots p_r\,b \not\equiv 0 \mod p_j\ .
\end{align*}
Hence $\gcd(x,ab) = 1$ and $x \in \Z/ab\Z$ is a generator with $\pi(x) = \bar{y}$.
\end{proof}

To start the induction we need the following lemma.

\begin{lemma} \label{lem:ind_start}
Let $G$ be a countable torsion free discrete group, which fits into an exact sequence
\[
\xymatrix{
	0 \ar[r] & \Z \ar[r] & G \ar[r] & \Z/m\Z \ar[r] & 0
}
\]
Then $G$ is isomorphic to $\Z$, hence in particular connective.
\end{lemma}

\begin{proof}
	This can be proven by calculating $H^2(\Z/m\Z,\Z)$ for all $\Z/m\Z$-module structures on $\Z$, but we give a direct argument here.
	
	Let $x \in G$ be a lift of $1 \in \Z/m\Z$. Then $G$ is generated by $x$ and $\Z$. Moreover, $x^m \in \Z \cap Z(G)$, where $Z(G)$ denotes the center of $G$. We have $\Aut{\Z} \cong GL_1(\Z) \cong \Z/2\Z$. If $t \in \Z$ denotes a generator, we therefore can only have $xtx^{-1} = t^{-1}$ or $xt = tx$. Suppose the first is true, then
	\[
		x^m = x\,x^m\,x^{-1} = x^{-m} \quad \Rightarrow \quad x^{2m} = e
	\]
	contradicting that $G$ is torsion free. Thus, $t$ and $x$ commute and $x^m = t^n$ for some $n \in \Z$. Without loss of generality we can assume $\gcd(m,n) = 1$. Indeed, if $m = m'\,\ell$ and $n = n'\,\ell$ with $\ell > 1$, then $(x^{m'}t^{-n'})^{\ell} = e$ and therefore $x^{m'} = t^{n'}$ also holds in $G$. Consider
	\[
		\alpha \colon G \to \Z \quad ; \quad x^kt^{\ell} \mapsto k\,n + \ell\,m\ .
	\]
	This is a well-defined group homomorphism, which is easily seen to be bijective as a consequence of $\gcd(m,n) = 1$.
\end{proof}

\begin{theorem} \label{thm:cyclic_holonomy}
Let $G$ be a countable, torsion free, discrete group, which fits into an exact sequence of the form
\[
	\xymatrix{
		0 \ar[r] & \Z^n \ar[r] & G \ar[r]^-{\pi} & \Z/m\Z \ar[r] & 0
	}
\]
for some $n,m \in \N$. Then $G$ is connective.
\end{theorem}

\begin{proof}
This will be proven by induction over the rank of the free abelian subgroup. The case $n=1$ follows from Lemma~\ref{lem:ind_start}.	

Observe that $Z(G) \neq \{e\}$. Indeed, let $x \in G$ be a lift of the generator of $\Z/m\Z$. Then $G$ is generated by $\Z^n$ and $x$. Moreover, $x^m \neq e$ since $G$ is torsion free and $\pi(x^m)$ is trivial, hence $x^m \in \Z^n$. Thus, $x^m$ commutes with $\Z^n$ and $x$, hence with all elements of $G$, i.e.\ $x^m \in Z(G)$.

This implies that the transfer homomorphism $T \colon G \to \Z^n$ associated to the finite index subgroup $\Z^n$ is non-trivial. Therefore there exists a surjective group homomorphism $\varphi \colon G \to \Z$. Let $q \colon \Z \to \Z/m\Z$ be the canonical quotient homomorphism and let $\bar{\varphi} = q \circ \varphi$. Let $H = \ker(\varphi)$. We have the following commutative diagram with exact rows and columns:
\[
\xymatrix{
	& & 0 \ar[d] & 0 \ar[d] \\
	& & H \ar[r] \ar[d] & \Z/a\Z \ar[r] \ar[d] & 0 \\
 	0 \ar[r] & \Z^n \ar[d] \ar[r] & G \ar[r]^-{\pi} \ar[d]^{\varphi} & \Z/m\Z \ar[r] \ar[d] & 0 \\
 	0 \ar[r] & b\Z \ar[r] & \Z \ar[r]^-{\pi'} \ar[d] & \Z/b\Z \ar[r] \ar[d] & 0 \\
	& & 0 & 0
}
\]
The value of $a$ is chosen in such a way that $\text{Im}(\left.\pi\right|_H) \cong \Z/a\Z$ and $b$ satisfies $m = ab$. The homomorphism $\pi'$ is surjective since $\pi$ and $\Z/m\Z \to \Z/b\Z$ are. The vertical arrow on the left hand side is induced by $\left.\varphi\right|_{\Z^n}$.

Suppose $H \subset \ker(\pi) = \Z^n$. Then $a = 1$, $b =m$ and $\pi = \pi' \circ \varphi$. By Lemma~\ref{lem:fin_ext}, $G$ is then connective. So we may assume $a > 1$.

We claim that there is an element $g \in G$ such that $\varphi(g) = 1$ and $\pi(g)$ is a generator of $\Z/m\Z$. This is constructed as follows: If $b = 1$, we can choose $g \in G$, such that $\varphi(g) = 1$ and modify it by an element in $H$ to achieve that $\pi(g)$ becomes a generator. Otherwise, choose $g' \in G$ such that $\varphi(g') = 1$ and note that $\pi'(\varphi(g'))$ is a generator of $\Z/b\Z$ by surjectivity. We can lift $\pi'(\varphi(g'))$ to a generator $x \in \Z/m\Z$ by Lemma~\ref{lem:gen_lift}. Note that $\pi(g') - x \in \Z/a\Z$ and lift this difference to an element $h \in H$. Let $g = g'\,h^{-1}$. Then $\varphi(g) = \varphi(g') = 1$ and $\pi(g) = \pi(g') - \pi(h) = \pi(g') - \pi(g') + x = x$.

Let $G' = \ker(\bar{\varphi}) = \{g \in G \ |\ \varphi(g) = \ell \cdot m \text{ for } \ell \in \Z\} \supset H$. Hence, the following diagram has exact rows.
\begin{equation} \label{eqn:decomposition}
\xymatrix{
	1 \ar[r] & G' \ar[d]_-{\left.\varphi\right|_{G'}} \ar[r] & G \ar[r]^-{\bar{\varphi}} \ar[d]^{\varphi} & \Z/m\Z \ar[r] \ar[d]_{=} & 0 \\
	1 \ar[r] & m\Z \ar[r] & \Z \ar[r]_-{q} & \Z/m\Z \ar[r] & 0 \\	
}
\end{equation}
The group $G$ is generated by $\Z^n$ and the element $g$ constructed above. We have $g^m \in \Z^n \cap Z(G)$ and $\varphi(g^m) = m$. In particular, $g^m \in Z(G')$. Let
\[
	\psi \colon H \times \Z \to G' \quad ; \quad (h,k) \mapsto h\cdot g^{mk}\ .
\]
This is a group homomorphism, since $g^m$ is central and it fits into the commutative diagram with exact upper and lower part
\[
\xymatrix{
	& & G'\ar[dr]^{\left.\varphi\right|_{G'}/m} \\
	1 \ar[r] & H \ar[dr]_-{i} \ar[ur] &  & \Z \ar[r] & 0\\
	& & H \times \Z \ar[ur]_-{\text{pr}} \ar[uu]_-{\psi}
}
\]
proving that $\psi$ is in fact an isomorphism. By the upper row in diagram (\ref{eqn:decomposition}) and Lemma~\ref{lem:fin_ext}, the connectivity of $G$ follows if $H \times \Z$, hence $H$, is connective \cite[Thm.~4.1]{Dad-Pennig-homotopy-symm}. But $H$ fits into a short exact sequence of the form
\[
	0 \to A \to H \to \Z/a\Z \to 0
\]
where $A$ is the free abelian kernel of the nonzero homomorphism $\Z^n \to b\Z$ from above, which has rank $(n-1)$. This completes the induction step.
\end{proof}
\section{Connectivity of Lie group $C^*$-algebras}
In this section we determine which linear connected nilpotent Lie groups  and which linear connected reductive Lie groups have connective reduced $C^*$-algebras. Let us recall that nilpotent connected Lie groups are liminary as shown by Dixmier \cite{Dixmier-nilpotent} and Kirillov \cite{Kirillov-nilpotent} and semisimple connected Lie groups are liminary as shown by Harish-Chandra \cite{Harish-Chandra-CCR}.
\subsection{Solvable and nilpotent Lie groups}

A locally compact group $N$ is compactly generated if $N=\bigcup_n V^n$ for some compact subset $V$ of $N$. Every connected locally compact group is automatically compactly generated.
The structure of abelian compactly generated locally compact groups is known. If $N$ is such a group, then
$N\cong \R^n \times \Z^m \times K$ for integers $n,m\geq 0$ and $K$ a compact group, \cite[Thm.\ 4.4.2]{Deitmar-Echterhoff}.
\begin{proposition}\label{prop:center}
If $G$ is a second countable locally compact amenable group (for example a solvable Lie group) whose  center contains a noncompact closed  connected  subgroup,
then $C^*(G)$ is connective.
\end{proposition}
\begin{proof}
Let $N$ be a noncompact closed  connected  subgroup of $Z(G)$. Then, by the structure theorem quoted above, $N$ must have a closed subgroup isomorphic to $\R$.
Consider the central extension:
\[0 \to \R\to G \to H \to 0. \]
   Since $G$ is amenable, by \cite[Thm. 1.2]{Packer-Raeburn} (as explained in \cite[Lemma 6.3]{Echterhoff-Williams}), $C^*(G)$ has the structure of a continuous field of $C^*$-algebras over $\widehat{\R}\cong \R$.
  The desired conclusion follows from Cor.~\ref{cor:fields} since $\R$ has no compact open subsets.
\end{proof}

\begin{example}
We give here two examples  that complement Proposition~\ref{prop:center}.
\begin{itemize}

\item[(i)] Simply connected solvable Lie groups can have discrete noncompact centers.
This is the case for $G=\C\rtimes_\alpha \R$ where $\alpha:\R \to \mathrm{Aut}(\C)$ is defined by $\alpha(t)(z)=e^{it}z$ for $t\in \R$ and $z\in \C$. In this case $Z(G)=\{0\}\times 2\pi \Z$.

Nevertheless in this case $C^*(G)$ is connective.
Consider the extension \[0\to Z(G)\to G \to G/Z(G)\cong \C \times \T\to 0.\]
Then $C^*(G)$ is a continuous $C(\T)$-algebra whose fiber at $1$ is the algebra $C^*(\C \times \T)$. Since
$C^*(\C \times \T)\cong C_0(\R^2)\otimes c_0(\Z)$ is connective, so is $C^*(G)$.

\item[(ii)] Both the real  and the complex ``$ax+b$'' groups
\[G=\left\{\begin{pmatrix} a & b \\ 0 & 1\end{pmatrix}\colon a\in F^{\times}, b\in F\right\}\]
where $F=\R$ or $F=\C$
 are solvable with trivial center and their $C^*$-algebras contain a copy of the compacts $\K$, see \cite{Rosenberg-solvable}, and so they are not connective.
\end{itemize}
\end{example}

\begin{theorem}\label{thm:nilpotent} Let $G$ be a (real or complex) linear connected nilpotent Lie group. Then $C^*(G)$ is connective
if and only if $G$ is not compact.
\end{theorem}
\begin{proof} We view $G$ as a real Lie group.
By \cite[Chap.~2, Thm.~7.3]{Vinberg-ency}, if $G$ is a linear connected nilpotent Lie group, then $G$   decomposes as a direct product $G=T\times N$ of a torus $T$ and a simply connected nilpotent group $N$.
If $G$ is compact, then $G=T$ and $C^*(G)$ is isomorphic to a direct sum of $\C$ so that it is not connective.
If $G$ is noncompact, then $N$ is nontrivial and so the center of $G$ is given by $Z(G)=T\times Z(N)$,
where the center $Z(N)$ of $N$ is isomorphic to $\R^n$ for some $n\geq 1$.
We conclude the proof by applying Proposition~\ref{prop:center}.
\end{proof}

\begin{remark} It is not true that  a liminary (CCR) $C^*$-algebra  is connective if and only if does not have nonzero projections. Indeed
\[A=\left\{f\in C([0,1], M_2(\C))\colon f(0)=\begin{pmatrix}\lambda & 0\\ 0 & \lambda \end{pmatrix},\, f(1)=\begin{pmatrix}\lambda & 0\\ 0 & 0\end{pmatrix},\, \lambda\in\C\right\} \]
does not contain nonzero projections but is not connective since $\mathrm{Prim}(A)$ is homeomorphic
to a circle $S^1$ and hence it is compact (and open in itself).

\end{remark}
\subsection{Reductive Lie groups}

A linear connected reductive group $G$ is a closed group of real or complex matrices that is closed under conjugate transpose. In other words $G$ is a closed and selfadjoint subgroup of the general linear group over
either $\R$ or $\C$.
 A linear connected semisimple  group is a linear connected reductive group with finite center \cite{Knapp-book}.

Say $G\subset GL(n,\R)$ or $GL(n,\C)$.
Define
$K=G\cap O(n)$, or $K=G\cap U(n)$ in the complex case.
If $G$ is linear connected reductive, then $K$ is compact, connected and is a maximal compact subgroup of $G$
\cite[Prop.1.2]{Knapp-book}.

Let $G=KAN$ be the  Iwasawa decomposition of the  linear connected semisimple Lie group $G$. $A$ is abelian and $N$ is nilpotent and both are closed simply connected subgroups of $G$, \cite[Thm. 5.12]{Knapp-book}.

First we consider the case of complex Lie groups.

\begin{theorem} \label{thm:complexLie} If $G$ is a linear connected  complex semisimple Lie group, then
 $C^*_r(G)$ is connective if and only if $G$ is not compact.
\end{theorem}
\begin{proof} If $G$ is compact, $C^*_r(G)$  is isomorphic to a direct sum of matrix algebras and hence it is not connective as it contains nonzero projections.

Conversely, suppose now that $G$ is a non-compact linear connected semi\-simple complex Lie group.
Note that from the Cartan decomposition $G=KAK$ \cite[Thm. 5.20]{Knapp-book} it follows that since $G$ is non-compact,  so is $A$ and therefore $A\cong \R^n$, for some $n\geq 1$.

 Let $M$ be the centralizer of $A$ in $K$.
By Lemma 3.3 and Proposition 4.1 of \cite{Penington-Plymen},
it follows that
\[C^*_r(G)\subset C_0(\widehat{M}\times \widehat{A},\mathcal{K}).\]
Since $\widehat{M}\times \widehat{A}\cong \widehat{M}\times \R^n$ does not have nonempty compact open subsets, it follows from Proposition~\ref{prop:compact-open}(ii) that $C_0(\widehat{M}\times \widehat{A},\mathcal{K})$
is connective. This completes the proof since connectivity passes to $C^*$-subalgebras.
\end{proof}
Next we consider the case of linear connected  real reductive Lie groups.
An element $g\in G$ is semisimple if it can be diagonalized over $\mathbb{C}$ when viewed as a matrix
$g\in M_n(\mathbb{C})$.

A closed subgroup $H$ of $G$ is called a Cartan subgroup if it is a maximal abelian subgroup consisting of semisimple elements, \cite[p.67]{Herb}.
If $G$ is either compact or a complex Lie group, then all Cartan subgroups of $G$ are connected and they are conjugated inside $G$.
In the general case $G$ has finitely many Cartan subgroups up to conjugacy and Cartan subgroups can have finitely many connected components.


We denote by $\widehat{G}_d\subset \widehat{G}$ the discrete series representations.
It consists of unitary equivalence classes of square-integrable representations \[\sigma:G \to U(H_\sigma).\]
 Harish-Chandra has shown that the discrete series representations of a semi-simple Lie group $G$
 are parametrized  by compact Cartan subgroups and in particular $G$ has discrete series representations
 if and only if it has a compact Cartan subgroup, \cite{Harish-Chandra-discrete1, Harish-Chandra-discrete2}.

We recall the following facts from \cite[p.72]{Herb} concerning cuspidal parabolic subgroups.
 Let $H$ be a Cartan subgroup of $G$. Then $H$ decomposes as a direct product
$H=TA=T\times A$,
where $T$ is an abelian compact group and $A$ is a vector group isomorphic to $\R^n$ for $n\geq 0$.
The case $n=0$ occurs when $H$ is a compact Cartan subgroup.
The centralizer of $A$ in $G$ denoted by
\[L=C_G(A)=\{g\in G\colon ga=ag,\, \forall a \in A\}\]
 is a Levi subgroup of $G$. This means that there is a parabolic subgroup of $G$ of the form $P=LN$ (not unique) with $L$ as Levi subgroup.
Since $A$ is central in $L$, $H$ is a relatively compact Cartan subgroup of $L$, i.e.
$H/Z(L)$ is compact. This implies that $L$ has discrete series representations. Such a parabolic subgroup $P=LN$ is called  cuspidal.

One can further decompose $L=MA$ to obtain a Langlands decomposition
\[P=MAN=MA \ltimes N,\]
with $N$ a unipotent group.
If $H$ is a compact Cartan subgroup, then $L=P=G$ by \cite[p.72]{Herb}.

 We will write $P=M_PA_P N_P$ whenever we want emphasize the components of $P$.

The description of $C^*_r(G)$ relies on the analysis of
  the unitary principal series representations of $G$ associated to parabolic cuspidal subgroups $P$ (also called the $P$-principal series). They are of the form
\[\mathrm{Ind}_P^G(\sigma \otimes \omega \otimes 1_N),\]
where $\sigma\in \widehat{M}_d$ and $\omega\in \widehat{A}$ and $1_N$ is the trivial representation of $N$.

Consider two pairs $(P_i,\sigma_i)$, $i=1,2$ consisting of cuspidal parabolic subgroups of $G$ and irreducible square-integrable unitary representations of the subgroups $M_i$, $i=1.2$. We say that the pairs are \emph{associated} if there is $g\in G$ such that $gP_1g^{-1}=P_2$ and $\sigma_1(g \cdot g^{-1})$ is unitarily equivalent to $\sigma_2$.
This is an equivalence relation \cite[Def.~5.2]{ClareCrispHigson}. We denote by $[P,\sigma]$ the equivalence class of the pair $(P,\sigma)$.

The following statement is based on the calculation of $C^*_r(G)$ by A.~Wassermann \cite{A-Wassemann} although we don't really use the full strength of his results. An expanded treatment of the structure of $C^*_r(G)$ appears in \cite{ClareCrispHigson}.

Let $J(G)=\bigcap_{\pi \in \widehat{G}_d} \mathrm{ker}(\pi)\subset C^*_r(G)$ be the common kernel of the discrete series representations.
The following theorem shows that the K-homology of  $C^*_r(G)$ can be described in terms of homotopy classes asymptotic morphisms $C^*_r(G)\to \K$ which factor through $J(G)$ and
discrete series representations.

\begin{theorem}\label{thm:real_reductive}
Let  $G$ be a linear connected real reductive Lie group. Then
$C^*_r(G)\cong J(G)\oplus \bigoplus_{\sigma\in \widehat{G}_d}K(H_\sigma)$ and $J(G)$ is a connective liminary $C^*$-algebra. Moreover,
the following assertions are equivalent:
\begin{itemize}
  \item[(i)] $C^*_r(G)$ is connective,
  \item[(ii)] $G$ does not have discrete series representations,
  \item[(iii)] $G$ does not have a compact Cartan subgroup,
  \item[(iv)] There are no nonzero projections in $C_r^*(G)$.
\end{itemize}
\end{theorem}
\begin{proof}
As explained in \cite[p.560]{A-Wassemann}, \cite[p.1306]{ClareCrispHigson} the reduced $C^*$-algebra of a linear reductive connected Lie group admits an embedding
\[C^*_r(G)\hookrightarrow \bigoplus_{[P,\sigma]} C_0(\widehat{A}_P, \K(H_\sigma)), \]
where the  direct sum is over equivalence classes $[P,\sigma]$ as above.
It is important to emphasize that
if $G$ has a compact Cartan subgroup, then $G$ itself is one of the cuspidal parabolic subgroups and we have:
\[C^*_r(G)\hookrightarrow  \bigoplus_{[P,\sigma]} C_0(\widehat{A}_P, \K(H_\sigma))\oplus \bigoplus_{\sigma\in \widehat{G}_d} \K(H_\sigma), \]
where the first direct sum involves proper cuspidal parabolic subgroups $P=M_PA_P N_P$ and hence $\mathrm{dim}(\widehat{A}_P)>0$.
Moreover  by \cite{A-Wassemann}, \cite{ClareCrispHigson}:
\begin{equation}\label{eqn:decomp} C^*_r(G)\cong J(G)\oplus \bigoplus_{\sigma\in \widehat{G}_d}\K(H_\sigma)\end{equation}
 where
\[
	J(G) \hookrightarrow \bigoplus_{[P,\sigma]} C_0(\widehat{A}_P, \K(H_\sigma))\ .
\]
Hence, $J(G)$ is connective being a subalgebra of a connective $C^*$-algebra. The first part of the statement follows now from the decomposition \eqref{eqn:decomp}.

The equivalence $(ii) \Leftrightarrow (iii)$ is Harish-Chandra's result mentioned earlier.
In view of the decomposition \eqref{eqn:decomp}, $(ii)$ implies that  $C^*_r(G)=J(G)$ and hence $(i)$ since $J(G)$ is always connective. Connective $C^*$-algebras do not contain nonzero projections and hence $(i) \Rightarrow (iv)$.
Finally by using \eqref{eqn:decomp} again, we see that $(iv) \Rightarrow (ii)$ since $\K(H_\sigma)$ contains nonzero projections if $H_\sigma\neq 0$.
\end{proof}
\subsection{A remark on full $C^*$-algebras of Lie groups}

The full $C^*$-algebra $C^*(G)$ of a property $\mathrm{(T)}$ Lie group $G$ contains  nonzero projections and hence it is not connective, see \cite{Valette:full}. Nevertheless, inspection of several classes of examples indicates that  $C^*(G)$ has  interesting
connective  ideals that arise naturally from the representation theory of $G$. We postpone a detailed discussion of what is known for another time, but would like to mention two examples.

If $G$ is a connected semisimple Lie group with finite center, then $C^*(G)$ is liminary (or CCR), see \cite[p.115]{Wolf:book_harmonic}.

\begin{proposition} (a) $C^*(SL_2(\C))$ is connective.

(b) $C^*(SL_3(\C))=I(SL_3(\C))\oplus \C$ and $I(SL_3(\C))$ is connective.
\end{proposition}
\begin{proof} (a) $C^*(SL_2(\C))$ was computed by Fell \cite[Thm.\ 5.4]{Fell}. We describe now his result.
Let $Z$ be the subspace of $\R^2$ defined by  $Z=\bigcup_{n=0}^\infty \{n\}\times L_n$
where $L_0=(-1,\infty)$ and $L_n=(-\infty,\infty)$ for all $n\geq 1$. Endow $Z$ with the induced topology from $\R^2$.
Let $H_0$ be a separable infinite dimensional Hilbert space, let $H=H_0\oplus \C$ and fix a unitary operator $V:H_0\to H$.
Then $C^*(SL_2(\C))$ is isomorphic to
\[
 \{F\in C_0(Z,\K(H))\colon  F(0,-1)=V^*F(2,0)V\oplus \lambda, \,\text{for some} \,\, \lambda \in \C\}
\]
Since $Z$ has no nonempty open compact subsets it follows that $C_0(Z,\K(H))$ is connective and therefore so is its subalgebra $C^*(SL_2(\C))$.

(b)  This will be obtained as a consequence of the following result on the structure of $C^*(SL_3(\C))$ obtained by Pierrot \cite{Pierrot}.
Let $G=SL_3(\C)$ and denote by $\lambda_G:C^*(G)\to C_r^*(G)$ the morphism induced by the left regular representation and by $\iota_G:C^*(G)\to \C$ the trivial representation.
Pierrot proved that the kernel $J$ of the morphism $\lambda_G\oplus \iota_G :C^*(G)\to C_r^*(G) \oplus \C$
is a contractible $C^*$-algebra. The representation $\iota_G$ is isolated since $G$ has property $\mathrm{(T)}$.
Therefore there is an exact sequence
\[0 \to J \to I(G) \to C_r^*(G) \to 0\]
where $J$ is contractible and $C_r^*(G)$ is connective by Theorem~\ref{thm:complexLie}.
We conclude that $I(G)$ is connective by applying Theorem~\ref{thm:extensions}.

\end{proof}

\textbf{Acknowledgements}
We would like to thank Andrzej Szczepa\'{n}ski for pointing out a glitch in the formulation of Definition \ref{def:shielded}.
%
%
\bibliographystyle{abbrvnat}

\end{document}